\documentclass[10pt]{article}
\usepackage{amsmath,amssymb,amsthm}
\usepackage{epsfig}
\usepackage{color}

\title{$(k,k',k'')$-domination in graphs}

\author {
Abdollah Khodkar\thanks{corresponding author} \\
Department of Mathematics\\
University of West Georgia\\
Carrollton, GA 30118, USA\\
{\tt akhodkar@westga.edu}\vspace{3mm}\\
Babak Samadi\\
Department of  Mathematics\\
Arak University, Arak IRI\\
{\tt b-samadi@araku.ac.ir}\vspace{3mm}\\
H.R. Golmohammadi\\
Department of Mathematics\\
University of Tafresh, Tafresh, IRI\\
{\tt h.golmohamadi@tafreshu.ac.ir}\vspace{3mm}\\
}

\date{}

 \newtheorem{theorem}{Theorem}[section]
\newtheorem{corollary}[theorem]{Corollary}

\theoremstyle{definition}

\begin{document}

\maketitle
\begin{abstract}
\noindent We first introduce the concept of $(k,k',k'')$-domination numbers in graphs, which is
a genaralization of many domination parameters. Then we find lower and upper bounds for this parameter,
which improve many well-known results in literatures.
\vspace{3mm}
\end{abstract}

%%%%%%%%%%%%%%%%%%%%%%%%%%%%%%%%%%%%%%%%%%%%%%%%%%%%%%%%%%%%%%%%%%%%
\section{Introduction and preliminaries}
Throughout this paper, let $G$ be a finite connected graph with vertex set $V=V(G)$, edge set $E=E(G)$, minimum degree $\delta=\delta(G)$ and maximum degree $\Delta=\Delta(G)$. We use \cite{w} as a reference for terminology and notation which are not defined here. For any vertex $v \in V$, $N(v)=\{u\in G\mid uv\in E(G)\}$ denotes the {\em open neighbourhood} of $v$ in $G$, and $N[v]=N(v)\cup \{v\}$ denotes its {\em closed neighbourhood}.

\noindent There are many domination parameters in graph theory. The diversity of domination parameters and the types of proofs involved are very extensive. We believe that some of the results in this field are similar and the main ideas of their proofs are the same. Therefore we introduce and investigate the concept of $(k,k',k'')$-domination number, as a generalization of many domination parameters, by a simple uniform approach.

\noindent Let $k,k'$ and $k''$ be nonnegative integers. A set  $S\subseteq V$ is a {\em $(k,k',k'')$-dominating set} in $G$ if every vertex in $S$ has at least $k$ neighbors in $S$ and every vertex in $V‎\setminus S‎$ has at least $k'$ neighbors in $S$ and at least $k''$ neighbors in $V‎\setminus S‎$. The {\em $(k,k',k'')$-domination number} $\gamma‎_{(k,k',k'')}‎(G)$ is the minimum cardinality of a $(k,k',k'')$-dominating set. We note that every graph with the minimum degree at least $k$ has a $(k,k',k'')$-dominating set, since $S=V(G)$ is such a set.
Note that
\begin{itemize}
\item $\gamma‎_{(0,1,1)}‎(G)=\gamma‎_{r}(G)$: {\em Restrained domination number};
\item $\gamma‎_{(1,1,1)}‎(G)= ‎\gamma‎‎‎_{t}^r‎‎(G)$: {\em Total restrained domination number};‎
\item $\gamma‎_{(1,2,1)}‎(G)=‎\gamma‎_{2r}‎‎(G)$: {\em Restrained double domination number};
\item  $\gamma‎_{(k,k,k)}‎(G)=\gamma‎‎‎_{‎\times k,t‎}^r(G): k$-{\em Tuple total restrained domination number};
\item $\gamma‎_{(k,k,0)}‎(G)=‎\gamma‎_{‎\times k,t‎}‎‎(G): k$-{\em Tuple total domination number};
\item $\gamma‎_{(k-1,k,0)}‎(G)=\gamma‎‎‎_‎_{‎\times ‎k}‎(G): k$-{\em Tuple domination number};
\item $\gamma‎_{(0,k,0)}‎(G)=‎\gamma‎_{k}(G)‎‎: k$-{\em Domination number}.
\end{itemize}
For the definitions of the parameters above and a comprehensive work on domination in graphs see \cite{cfhv,cr,dhhm,hhs,hhs2,kn,k}.

%%%%%%%%%%%%%%%%%%%%%%%%%%%%%%%%%%%%%%%%%%%%%%%%%%%
\section {Lower bounds on $(k,k',k'')$-domination numbers}

In this section, we calculate a lower bound on $\gamma‎_{(k,k',k'')}‎(G)$, which improves
the existing lower bounds on these seven parameters.

\noindent The following result can be found in \cite{cr} and \cite{hjjp2}.
\begin{theorem}\label{LB:TR}
If $G$ is a graph without isolated vertices of order $n$ and size $m$, then
\begin{equation}\label{EQ3}
 ‎\gamma‎‎‎_{t}^r‎‎(G)‎\geq 3n/2-m‎,
\end{equation}
in addition this bound is sharp.
\end{theorem}

\noindent Also Hattingh et.al \cite{hjlpv} found that
\begin{equation}\label{EQ4}
‎\gamma‎_{r}(G)‎‎‎\geq n-2m/3‎.
\end{equation}

The following known
result is an immediate consequence of Theorem \ref{LB:TR}.
\begin{theorem}\cite{hjjp1}
If $T$ is a tree of order $n‎\geq2‎$, then
\begin{equation}\label{EQ5}
 ‎\gamma‎‎‎_{t}^r‎‎(T)‎\geq ‎\lceil‎\frac{n+2}{2}‎‎\rceil‎‎‎.
 \end{equation}
\end{theorem}

The inequality
\begin{equation}\label{EQ6}
‎\gamma‎_{r}‎‎(T)‎\geq ‎\lceil‎\frac{n+2}{3}‎‎\rceil‎‎‎
\end{equation}
on restrained domination number of a tree of order $n‎\geq1‎$
was obtained by Domke et al. \cite{dhhm}.

\noindent The author in \cite{k} generalized Theorem \ref{LB:TR} and proved
that if $‎\delta(G)‎\geq k‎‎$, then
\begin{equation}\label{EQ7}
‎\gamma‎‎‎_{‎\times k,t‎}^r(G)‎‎‎\geq 3n/2-m‎/k.
\end{equation}
Moreover the authors in \cite{kn} proved that if $G$ is a graph without isolated vertices, then
\begin{equation}\label{EQ8}
‎\gamma‎_{2r}(G)‎‎‎\geq ‎\frac{5n-2m}{4}.
\end{equation}

\noindent We now improve the lower bounds given in
$(\ref{EQ3}), (\ref{EQ4}), \ldots,(\ref{EQ8})$.
For this purpose we first introduce a notation.
Let $G$ be a graph with $‎\delta(G)‎\geq k‎‎$ and let $S$ be a $(k,k',k'')$-dominating set in $G$. We define
$$‎\delta‎‎^{*}‎‎‎=\min\{\deg(v)\mid v‎\in V(G) \mbox{ and } \deg(v) \geq k'+k'' \}.$$

\noindent It is easy to see that $\deg(v)‎‎$ is at least $k'+k''$ and
therefore is at least $‎\delta‎^{*}‎‎$ for all vertices in $V‎\setminus‎ S$.

\begin{theorem}\label{Th:Lower.Bound}
Let $G$ be a graph with $‎\delta(G)‎\geq k‎‎$. Then
$$‎\gamma‎‎‎_{(k,k',k'')}(G)‎‎‎\geq ‎\frac{(k'+‎\delta‎^{*}‎‎)n-2m}{‎\delta‎^{*}+k'-k‎‎}.‎$$
\end{theorem}

\begin{proof}
Let $S$ be a minimum $(k,k',k'')$-dominating set in $G$. Then,
every vertex $v\in S$ is adjacent to at least $k$ vertices in $S$.
Therefore $|E(G[S])|‎\geq k|S|/2‎$. Let $E(v)$ be the set of edges at
vertex $v$. Now let $v‎\in V‎\setminus ‎S‎$. Since $S$ is a
$(k,k',k'')$-dominating set, it follows that $v$ is incident to at least $k'$ edges
$e‎_{1}, \ldots ,‎e‎_{k'}‎$ in $[S,V‎\setminus S‎]$ and at least $k''$
edges $e‎_{k'+1}, \ldots ,e‎_{k'+k''}‎‎$ in $E(G[V‎\setminus ‎S])‎$.
Since $deg(v)‎\geq ‎\delta‎^{*}‎\geq k'+k''‎‎‎‎$, $v$ is incident to at least
$‎\delta‎^{*}-k'-k''$ edges in $E(v)‎\setminus \{e‎_{i}‎\}‎_{i=1}^{k'+k''}‎‎$.
The value of $|[S,V‎\setminus S‎]|+|E(G[V‎\setminus ‎S])‎|$ is minimized if
the edges in $E(v)‎\setminus \{e‎_{i}‎\}‎_{i=1}^{k'+k''}‎‎$
belong to $E(G[V‎\setminus ‎S])$. Therefore
$$\begin{array}{lcl}
2m&‎=&2|E(G[S])|‎+2|[S,V\setminus S]|+2|E(G[V\setminus S])|\\
&\geq &‎k|S|+2k'(n-|S|)‎+k‎‎''(n-|S|)+(‎\delta‎^{*}-k'-k''‎‎)(n-|S|).
\end{array}$$
This leads to $‎\gamma‎‎‎_{‎(k,k',k'')}(G)=|S|‎\geq ‎\frac{(k'+‎\delta‎^{*}‎‎)n-2m}{‎\delta‎^{*}+k'-k‎‎}‎$.
\end{proof}

\noindent We note that when $(k,k',k'')=(1,1,1)$, then Theorem \ref{Th:Lower.Bound} gives
improvements for inequalities (\ref{EQ3}) and (\ref{EQ5}). When $(k,k',k'')=(0,1,1)$,
then it will be improvements of its corresponding results given by (\ref{EQ4}) and (\ref{EQ6}).
Also, if $(k,k',k'')=(k,k,k)$, Theorem \ref{Th:Lower.Bound} improves (\ref{EQ7}) and
if $(k,k',k'')=(1,2,1)$, it improves (\ref{EQ8}).

As an immediate result of Theorem \ref{Th:Lower.Bound}, we conclude the following result of Hattingh and Joubert.

\begin{corollary}\cite{hj}
If $G$ is a cubic graph of order $n$, then $‎\gamma‎‎_{r}‎(G)‎‎‎\geq ‎\frac{n}{4}‎‎$.
\end{corollary}
Also, for total restrained and restrained double domination numbers of a cubic graph $G$, we obtain
$‎\gamma‎‎‎_{t}^r‎‎‎(G)‎\geq ‎\frac{n}{3}‎‎$ and $\gamma‎‎‎_{2r}‎‎‎(G)‎\geq ‎\frac{n}{2}$,
‎by Theorem \ref{Th:Lower.Bound}‎, respectively.

\noindent Since $‎\gamma‎_{‎\times‎k}(G)‎‎=‎\gamma‎_{(k-1,k,0)}(G)‎‎$, Theorem \ref{Th:Lower.Bound} is improvements of the following results.
\begin{theorem}\cite{hh}
Let $G$ be a graph of order $n$ and $‎\delta(G)‎\geq k-1‎‎$, then
$$\gamma‎_{‎\times‎k}(G)‎\geq ‎\frac{2kn-2m}{k+1}‎‎$$
and this bound is sharp.
\end{theorem}

\begin{theorem}\cite{zwx}
Let $G$ be a graph of order $n$ and size $m$ with minimum degree $‎\delta‎\geq k‎‎$. Then $‎\gamma‎_{‎\times‎k,t}‎‎(G)‎\geq‎2(n-\frac{m}{k})$ and this bound is sharp.
\end{theorem}
Theorem \ref{Th:Lower.Bound} is also an improvement of the following Theorem.

\begin{theorem}\cite{fj2}
If $G$ is a graph with $n$ vertices and $m$ edges, then $‎\gamma‎_{k}(G)‎\geq n-\frac{m}{k}‎‎‎$ for each $k‎\geq1‎$.
\end{theorem}

\noindent We note that every graph $G$ with $‎\delta(G)‎‎\geq k$ has
a $(k,k',0)$-dominating set such as $S=V(G)$
and therefore $‎\gamma‎_{(k,k',0)}(G)$ is well-defined when $‎\delta(G)‎‎\geq k$.

\begin{theorem}\label{Th:LB2}
If $G$ is a graph of order $n$ and $‎\delta(G)‎‎\geq k‎$, then $‎\gamma‎_{(k,k',0)}(G)‎\geq k'n/(‎\Delta+k'-k‎)‎‎‎$.
\end{theorem}

\begin{proof}
Let $S$ be a minimum $(k,k',0)$-dominating set in $G$. Then each vertex of $S$ is adjacent to at least $k$ vertices in $S$ and therefore to at most $‎\Delta-k‎$ vertices in $V‎\setminus S‎$, and so $|[S,V‎\setminus S‎]|‎\leq (‎\Delta-k)‎|S|$. On the other hand, every vertex of $V‎\setminus S‎$ has at least $k'$ neighbors in $S$, and so $k'(n-|S|)‎\leq ‎|[S,V‎\setminus S‎]|‎$. Consequently, $‎\gamma‎_{(k,k',0)}(G)‎=|S|‎\geq k'n/(‎\Delta+k'-k‎)‎$.
\end{proof}

\noindent The following corollaries are immediate results of Theorem \ref{Th:LB2}.

\begin{corollary}(\cite{hk,zwx})
If $G$ is a graph of minimum degree at least $k$, then $‎\gamma‎_{‎\times k,t}(G)‎\geq kn/‎\Delta‎‎‎‎$ and this bound is sharp.
\end{corollary}

\begin{corollary}\cite{hh}
If $G$ is a graph of order $n$ with $‎\delta(G)‎\geq k-1‎‎$, then $‎\gamma‎_{‎\times k}(G)‎\geq kn/(‎\Delta+1‎)$ and this bound is sharp.
\end{corollary}

\begin{corollary}\cite{fj1}
If $G$ is a graph of order $n$, then $‎\gamma‎_{k}(G)‎‎‎\geq kn/(‎\Delta+k‎)‎$ for every integer $k‎$‎.
\end{corollary}

%%%%%%%%%%%%%%%%%%%%%%%%%%%%%%%%%%%%%%%%%%%%%%%%%%%%%%%%%%%%%
\section {Upper bounds on $(k,k',1)$-domination numbers}

\noindent In this section we present an upper bound on $(k,k',1)$-domination numbers and list some of the existing upper bounds which can be derived from this upper bound.

\begin{theorem}
Let $G$ be a graph of order $n$ and let $k$ and $k'‎$ be positive integers.
\begin{enumerate}
\item  If $k'‎\geq k+1‎$ and $‎\delta‎‎\geq k'+1‎$, then
$‎\gamma‎‎_{(k,k',1)}(G)‎\leq n-‎\delta+ k'-1‎‎‎‎$.

\item If $k\geq k'$ and $‎\delta‎‎\geq k+2‎$, then
$‎\gamma‎‎_{(k,k',1)}(G)‎\leq n-‎\delta+ k‎‎‎‎$.
\end{enumerate}
The bounds given in Parts 1 and 2 are sharp.
\end{theorem}

\begin{proof}
Let $u$ be a vertex in $G$ with $\deg(u)=‎\delta‎$.

\noindent {\bf Proof of 1:} Suppose that $v‎_{1},\ldots,v‎_{k'}‎‎‎\in N(u)‎$. Since $‎\delta‎\geq k'+1$, it follows that $|N[u]‎\setminus \{v‎_{1},\ldots,v‎_{k'}\}‎|‎\geq2‎$ and therefore $N[u]‎\setminus \{v‎_{1},\ldots,v‎_{k'}\}$ is a nonempty set. Also, it is easy to see that the subgraph induced by $N[u]‎\setminus \{v‎_{1},\ldots,v‎_{k'}\}$ has no isolated vertices. Now let $S=V(G)‎\setminus (N[u]‎\setminus \{v‎_{1},\ldots,v‎_{k'}\})‎$. Let $v‎\in‎ N[u]\setminus \{v‎_{1},\ldots,v‎_{k'}\}$. Then $v$ can be joint to at most $‎\delta-k'‎$ vertices in $N[u]‎\setminus \{v‎_{1},\ldots,v‎_{k'}\}$. Thus $v$ has at least $k'$ neighbors in $S$. On the other hand, for every vertex $v$ in $S$ we have
$$|N(v)‎\cap S‎|=deg(v)-|N(v)‎\cap (N[u]‎\setminus \{v‎_{1},\ldots,v‎_{k'}\})‎|‎\geq deg(v)-‎\delta +k'-1‎‎‎\geq k.‎$$
Therefore $S$ is a $(k,k',1)$-dominating set in $G$. Hence,
$$
\gamma‎_{(k,k',1)}(G)‎‎‎‎\leq |S|‎=|V(G)‎\setminus (N[u]‎\setminus \{v‎_{1},\ldots,v‎_{k'}\})‎|=n-‎\delta+k'-1.
$$

\noindent {\bf Proof of 2:} Suppose that $v‎_{1},\ldots,v‎_{k+1}‎‎‎\in N(u)‎$. By assumption,
$|N[u]‎\setminus \{v‎_{1},\ldots,v‎_{k+1}\}‎|‎\geq 2‎$ and the subgraph induced by
$N[u]‎\setminus \{v‎_{1},\ldots,v‎_{k+1}\}$ has no isolated vertices. An argument similar to that described
in Part 1 shows that $S=V(G)‎\setminus (N[u]‎\setminus \{v‎_{1},\ldots,v‎_{k+1}\})‎$
is a $(k,k',1)$-dominating set in $G$. Therefore
$$
\gamma‎_{(k,k',1)}(G)‎‎‎‎\leq |S|‎=|V(G)‎\setminus (N[u]‎\setminus \{v‎_{1},\ldots,v‎_{k+1}\})‎|=n-‎\delta+k.
$$

\noindent It is easy to see that the upper bounds are sharp for the complete graph $K‎_{n}‎$, when $n‎\geq \max\{k,k'\}+3‎$.
\end{proof}
Considering Parts 1 and 2 of Theorem 3.1 we can see that
$$‎\gamma‎‎_{(k,k',1)}(G)‎\leq n-‎\delta+ max\{k,k'-1\},‎‎‎‎$$
when $‎\delta(G)‎‎\geq max\{k,k'\}+2‎$.\\

As an immediate consequence we conclude the following corollary.

\begin{corollary}
If $G$ is a graph of order $n$ and minimum degree $‎\delta‎\geq3‎‎$, then $\gamma‎_{2r}‎‎(G)‎\leq n-‎\delta+1‎‎$ and the bound is sharp.
\end{corollary}

\noindent The authors in \cite{kn} showed that $\gamma‎_{2r}(G)‎\leq n-2‎$ for every graph of order $n$ and minimum degree $‎\delta(G‎‎)‎‎\geq 3‎$. In fact, Corollary 3.2 gives an improvement for this bound. In addition, if $‎\delta‎\geq4‎‎$,
then the upper bound $n-2$ for $\gamma‎_{2r}(G)$ is not sharp.

%%%%%%%%%%%%%%%%%%%%%%%%%%%%%%%%%%%%%%%%%%%%%%%%%%%%%%%%%%%%%%%%%%%%%%%%%%

\end{document}